\documentclass[11pt,twoside]{amsart}
\usepackage{amsmath, amsthm, amscd, amsfonts, amssymb, graphicx, color}
\usepackage[bookmarksnumbered, plainpages]{hyperref}
\allowdisplaybreaks
\usepackage{mathtools}
\usepackage{faktor}

\DeclareMathOperator{\supp}{supp}
\DeclareMathOperator{\spec}{Spec}
\DeclareMathOperator{\im}{Im}
\DeclareMathOperator{\Ker}{Ker}
\DeclareMathOperator{\Hom}{Hom}
\DeclareMathOperator\SL{SL}
\renewcommand{\phi}{\varphi}
\renewcommand*{\th}{\textsuperscript{th}}

\DeclarePairedDelimiterX\hook[2]{\left[}{\right]}{#1 : #2}
\newcommand*{\field}[1]{\mathbf{#1}}
\newcommand*{\R}{\field{R}}
\newcommand*{\Z}{\field{Z}}

\newtheorem{thm}{Theorem}[section]
\newtheorem*{mainthm}{Main Theorem}

\newtheorem{defn}[thm]{Definition}

\numberwithin{equation}{section}
\def\pn{\par\noindent}

\begin{document}

\title{A Cheeger-Buser-type inequality on CW complexes}
\author{Grégoire Schneeberger}

\thanks{{\scriptsize
}}
\maketitle


\begin{abstract}  
We extend the definition of boundary expansion to CW complexes and prove a Cheeger-Buser-type relation between the spectral gap of the Laplacian and the boundary expansion of an orientable CW complex.
\end{abstract}

\vskip 0.2 true cm


\pagestyle{myheadings}
%


\bigskip
\bigskip


\section{\bf Introduction}
\vskip 0.4 true cm
The expander graphs have been a prolific field of research in the last four decades (see for example \cite{Lubotzky2012} for an excellent survey). For a graph $X$ with a vertex set $V$ the classical expansion constant (or \emph{Cheeger constant}) is defined by 
\begin{equation*}
h(X)\coloneqq \min \left\{ \frac{|\partial A|}{\min\left\{|A|, |A^c|\right\}} : \emptyset \subsetneq A \subsetneq V \right\}
\end{equation*}
where $\partial A$ is the set of edges with one vertex in $A$ and the other in $A^c$. A central result of this field is the Cheeger-Buser inequality, which describes the relation between the expansion and the spectrum of the Laplacian.
\begin{thm}[Cheeger-Buser inequality]
Let $X$ be a connected graph and $\lambda$ the first non-trivial eigenvalue of the Laplacian, then
\begin{equation*}
\frac{\lambda}{2} \leq h(X) \leq \sqrt{2\lambda d}
\end{equation*}
where $d$ is the maximal degree of a vertex.
\end{thm}
For more details see \cite[Theorem 2.4]{Hoory2006}.

In recent years, theories for expansion of higher dimensional simplicial complexes have emerged. The combinatorial definitions allow generalizations of the Cheeger-Buser inequality, see \cite{Golubev2019,Gundert2015,Parzanchevski2016}. Other results, like Expander Mixing Lemma or generalization of Alon-Boppana theorem, can be proved using this formalism, see \cite{Parzanchevski2013}. We can also use homology and cohomology with coefficients in $\Z/2\Z$ to define boundary expansion, see \cite{Steenbergen2014}, and coboundary expansion, see \cite{Dotterrer2013,Gromov2010,Gundert2012,Linial2006,Meshulam2009,Steenbergen2014}. A Cheeger-Buser-type inequality is proved for boundary expansion in \cite{Steenbergen2014}. The coboundary expansion has the advantage that it coincides with the standard Cheeger constant in the one dimensional case, but the Cheeger-Buser inequality does not stay true in higher dimension, see \cite{Gundert2012,Steenbergen2014} for counterexamples. Nevertheless there exist some indications that suggest a connection between these two notions, particularly for the Cheeger's part (the upper bound) which holds for Riemannian manifolds \cite{Cheeger1970}. 

Recall that one of the first explicit constructions of expander graphs used Cayley graphs of finite quotients of a group with Kazhdan's Property (T), see \cite[Chapter 3]{Lubotzky2012} and \cite[Chapter 6]{Bekka2008}. There exist higher dimensional objects which can be associated to groups in the same spirit, as, for example, Cayley complexes \cite[Chapter 3]{Lyndon2001} or presentation complexes \cite{Hatcher2002}. For example, the group $\SL_3(\Z)$ has Property (T) and any family of its finite quotients will give an expanding family of Cayley graphs. One can wonder whether the corresponding Cayley complexes are also expanding. More generally, it would be interesting to establish which properties of the group may imply high-dimensional expansion in its finite quotients. One technical issue that has to be addressed while proceeding with this program is that high dimensional expansion has been mainly defined and studied for simplicial complexes, while the higher-dimensional objects naturally associated to groups are typically CW-complexes. Working out the formalism of high dimensional expansion for CW complexes is the aim of the present note.

We will begin by recalling some classical definitions about CW complexes, groups of cochains with coefficients in an abelian group $G$, Laplacians and their spectra.  Then, considering cochains with coefficients in $\Z/2\Z$, we will introduce $h_n(X)$, the $n$\th{} boundary expansion constant for every integer $n$ and we will prove a Cheeger-Buser-type inequality in the same spirit as the original result when $n$ equal to the dimension of the complex. This Theorem generalizes \cite{Steenbergen2014}. 

\begin{mainthm} \label{orientable_lower_bound}
Let X be a regular CW complex of dimension $d$ and $\lambda_d$ the smallest non trivial eigenvalue of the d\th{} lower Laplacian, then 
\begin{enumerate}
\item If $X$ is orientable 
\begin{equation*}
\lambda_d(X) \leq h_d(X).
\end{equation*}
\item If the maximal degree of a $(d-1)$-cell is 2, then
\begin{equation*}
h_d(X) \leq \sqrt{2 m \lambda_d}
\end{equation*}
where $m = \max \left\{ \sum_\mu \left| [e_\lambda^d : e_\mu^{d-1}]\right| : e_\lambda^d \in X^d \right\}$.
\end{enumerate}
\end{mainthm}

%
%
%
\section{\bf Definitions}
\subsection{CW Complexes}
We will begin by fixing some definitions and notations about CW complexes that we will use in the following. All the details can be found in \cite{Massey1980}. 

A \emph{CW complex} $X$ is a topological space obtained inductively by gluing euclidean balls, called \emph{cells}, via continuous maps called \emph{attaching maps}. In what follows, all complexes will be regular, which means that the attaching maps are homeomorphisms on their images. The \emph{$n$-skeleton}, denoted by $X^n$, is the set of all the cells of dimension $n$ which are called the \emph{$n$-cells}. The \emph{dimension} of $X$ is the maximal dimension of a cell. 
%
%
\subsection{Cohomology with coefficients}

We will now define the cohomology groups with coefficients associated to a CW complex. All the details of this classic construction can be found in \cite{Massey1980}. 

The group of \emph{$n$-chains} $C_n(X)$ is the free abelian group generated by the $n$-cells:
\begin{equation*}
C_n(X) = \bigoplus_{e_\lambda^{n} \in X^n} \Z.
\end{equation*}

By using classical homology tools, that we will not detail here, these groups can be provided with a structure of chain complex using boundary operators ${\partial_n : C_{n}(X) \rightarrow C_{n-1}(X)}$. 
In each infinite cyclic summand in the direct sum above, one can choose between the cyclic generator and its inverse: $b_\lambda^n$ or $\bar{b}_\lambda^n$. This choice defines the orientation of the n-cell $e_\lambda^n$. The set $\{b_\lambda^n\}_\lambda$ forms a basis of $C_n(X)$. The boundary operator is completely determined by the values on basis elements:
\begin{equation*}
\partial_n(b_\lambda^n)=\sum_\mu \left[b_\lambda^n : b_\mu^{n-1}\right] b_\mu^{n-1}
\end{equation*}
where the coefficients $[b_\lambda^n : b_\mu^{n-1}]$ are integers called the \emph{incidence numbers} of the cells $e_\lambda^n$ and $e_\mu^{n-1}$ with respect to the chosen orientations. This incidence number can be thought intuitively as the number of times an $n-1$ cell appears in the boundary of an $n$ cell, the sign depending of the consistency of the chosen orientations. The \emph{degree} of an oriented cell $b_\lambda^n$ is the sum $\sum_\mu \left|[b_\mu^{n+1} : b_\lambda^n]\right|$.

As the gluing functions are homeomorphisms, the incidence numbers take values in $\{-1,0,1\}$. Two oriented cells $b_{\lambda_1}^n$ and $b_{\lambda_2}^n$ that have a common $(n-1)$-cell $b_\mu^{n-1}$ in their boundary are either \emph{dissimilarly oriented} if $[b_{\lambda_1}^n : b_\mu^{n-1}]=[b_{\lambda_2}^n : b_\mu^{n-1}]$ or \emph{similarly oriented} if $[b_{\lambda_1}^n : b_\mu^{n-1}] \neq [b_{\lambda_2}^n : b_\mu^{n-1}]$. If there exists an orientation on a $d$-dimensional CW complex such that all the $d$-cells are similarly oriented, $X$ is said to be \emph{orientable}. 

The coboundary operator is defined using these incidence numbers :
\begin{align*}
\delta_n : C_n(X) &\rightarrow C_{n+1}(X)	\\
b_\lambda^n	& \mapsto \sum_\mu [b_\mu^{n+1} : b_\lambda^n]b_\mu^{n+1}.
\end{align*}
We will often omit the indices of the operators.

\begin{defn}
Let $G$ be an abelian group and $X$ be a CW complex. The $n$-cochains group with coefficients in $G$ is the group of homomorphisms between the $n$-chains and $G$,
\begin{equation*}
C^n(X:G)\coloneqq \Hom(C_n(X),G).
\end{equation*}
\end{defn}
It follows from the definition that $f(\bar{b}_\lambda^n)=-f(b_\lambda^n)$ for all $f \in C^n(X:G)$ and $n \geq 1$. We define operators, also denoted by $\partial_n$ and $\delta_n$, between the cochains groups:
\begin{align*}
\partial_n f(e_\mu^{n-1}) =f(\delta_n (e_\mu^{n-1})) = \sum_\lambda [e_\lambda^n : e_\mu^{n-1}]f(e_\lambda^{n})	\\*
\delta_n f (e_\mu^{n+1})=f(\partial_n (e_\mu^{n+1})) = \sum_\lambda [e_\mu^{n+1} : e_\lambda^{n}]f(e_\lambda^{n})
\end{align*}

It is sometimes more convenient to add $(-1)$-chains consisting only of the empty set and operators
\begin{equation*}
\partial_0(\sum_\lambda g_\lambda e_\lambda^0)=\sum_\lambda g_\lambda \, \textnormal{ and } \, \delta_{-1}(g)=\sum_\lambda g e_\lambda^0.
\end{equation*}
The following subgroups will be used in the following: 
\begin{eqnarray*}
B_n \coloneqq \im \partial_{n+1}	\quad & Z_n \coloneqq \Ker \partial_n\\
\end{eqnarray*}
\subsection{Laplacians and eigenvalues}
Consider the case $G = \R$. The cochains $C^n(X:\R)$ can be turned into real Hilbert spaces using 
\begin{equation*}
\langle f,g \rangle = \sum_\lambda f(e_\lambda^n) g(e_\lambda^n).
\end{equation*}
In this case, $\partial$ and $\delta$ are adjoint operators. Combining them, we define the \emph{n\th{} lower Laplacian}, 
\begin{equation*}
\Delta_n^-=\delta_{n-1} \partial_{n}.
\end{equation*}

It can be noted that the elements of $B_n$ are in the kernel of $\Delta_n^-$. Indeed, if $f$ is in $B_n$, there exists $g$ in $C^{n+1}(X:\R)$ such that $f =  \partial_{n+1}g$ and then $\Delta^-_n f = \delta_{n-1} \partial_n f = \delta_{n-1} \partial_n \partial_{n+1} g =0$. This part will be called the \emph{trivial} part of the spectrum and we will be interested in the smallest eigenvalue on the other parts.
\begin{defn}
The smallest non trivial eigenvalue of $\Delta_n^-$, denoted by $\lambda_n$, is defined as
\begin{equation*}
\lambda_n = \min \spec \Delta_n^-|_{B_n^{\perp}} .
\end{equation*}
\end{defn}

It can be computed using Rayleigh's quotients:
\begin{align*}
\lambda_n & :=\min \left\{\frac{ \| \partial_n f \|^2}{\|f\|^2} : f \in B_n^\perp, f \neq 0 \right\} =\min \left\{\frac{ \| \partial_n f \|^2}{\|f + B_n|^2} : f \not\in B_n \right\}
\end{align*}
where $\|f + B_n \|=\min \{ \|f + g \| : g \in B_n \}$.

\subsection{Boundary expansion }

Let us consider $G = \Z/2\Z$ to define the notion of boundary expansion for CW complexes. The cochains groups can be endowed with the Hamming's norm: 
\begin{equation*}
\| \alpha \| = | \supp \alpha |
\end{equation*}
for $\alpha \in C^n(X:\Z/2\Z)$.
We can define the following notion of expansion for CW complexes.
\begin{defn}
Let $X$ be a CW complex. The $n$\th{} boundary expansion constant of $X$ is:
\begin{equation*}
h_n(X)\coloneqq \min\left\{ \frac{\|\partial \alpha \|}{\|\alpha + B_n\|} : \alpha \in C^n(X:\Z/2\Z) \setminus B_n \right\}.
\end{equation*}
where $\|\alpha + B_n \|=\min \{ \|\alpha + \beta \| : \beta \in B_n \}$
\end{defn}

\section{\bf Proof of the Theorem}

\begin{proof}[Proof of 1)]
Let $\alpha$ be an element of $C^d(X : \Z/2\Z)$ which realizes the minimum in $h_d$. We can find a cochain $f$ in $C^d(X : \R)$, which assigns 1 to every $d$-cells in $\supp \alpha$ and $0$ to all the others. Since $X$ is orientable, $\partial_d \alpha$ is equivalent to $\partial_d f$. Then,
\begin{align*}
h_d 	& = \frac{\| \partial_d \alpha\|}{\| \alpha \|}\\
 	& = \frac{\| \partial_d f\|_2^2}{\| f \|_2^2} \\
	&\geq \min\left\{ \frac{\| \partial_d g\|^2}{\| g \|^2} : g \in C^d(X : \R), g \notin B_d = 0 \right\} \\
	& =  \lambda_d.
\end{align*}
\end{proof}

\begin{proof}[Proof of 2)]
Let $f$ be a real cochain which is an eigenvector of $\Delta_d^-$ of eigenvalue $\lambda_d$. We chose an orientation on the $d$-cells such that all the values of $f$ are positive. We do not assume that they are similarly oriented. We put an order on $X^d=\{ e_1^d, e_2^d, \ldots, e_N^d \}$ such that 
\begin{equation*}
0 \leq f(e_1^d) \leq f(e_2^d) \leq \ldots \leq f(e_N^d).
\end{equation*}
The boundary of $X$ is the $(d-1)$-cells with degree $1$,
\begin{equation*}
\partial X := \{ e_\lambda^{d-1} \in X^{d-1} : \deg e_\lambda^{d-1}  = 1 \}
\end{equation*}
For each $e_\lambda^{d-1}$ in $\partial X$, we add another $(d-1)$-cell ,$e_{\lambda'}^{d-1}$, via the attaching map $\varphi_{\lambda'} = \varphi_\lambda$. We can add a $d$-cell on $X$, whose attaching map goes homeomorphically into $e_\lambda^{d-1} \cup e_{\lambda'}^{d-1}$. We denote by $X_\partial^d$ the set of these new $d$-cells and put an order $X_\partial^d = \{e_{0}^d, e_{-1}^d,\ldots, e_{1-M}^d \}$ on it. The function $f$ is defined as $f=0$ on the cells of $X_\partial^d$. When two $d$-cells have a common $(d-1)$-cell in their boundary, we say they are \emph{low adjacent} and write $e_\lambda^d \sim e_\lambda^d$. It is possible that there are more than one $(d-1)$-cells in the intersection of the boundary of two $d$-cells. We say that we count the cells that realize $e_\lambda^d \sim e_\lambda^d$ with \emph{multiplicity} in this case, i.e. the pair $\{e_j^d, e_k^d \}$ appears a number of time equal of the number of common $(d-1)$-cells in their boundary. We define
\begin{equation*}
C_i:=\left\{ \{e_j^d, e_k^d \} : 1-M \leq j \leq i < k \leq N \text{ and } e_j^d \sim e_k^d   \right\}
\end{equation*} 
counted with multiplicity. Consider the quantity
\begin{equation*}
H[f]:=\min_{0 \leq i \leq N-1} \cfrac{|C_i|}{N-i}. 
\end{equation*}
We can show that $H[f] \geq h_d$. Indeed, let $\mathbf{i}$ be the $i$ which realizes the minimum of $H[f]$ and $\alpha \in C(X:\Z/2\Z)$ defined as follow,
\begin{equation*}
\alpha(e_k^d)= 
\begin{cases}
1 & \mathbf{i}<k	\\
0 & \mathbf{i} \geq k 
\end{cases}.
\end{equation*}
So the $(d-1)$-cells that are in the support of $\partial_d \alpha$ are in the boundary of one $e_k^d$ with $k>i$ and another with $k \leq i$. Then, we have
\begin{align*}
H[f]=\frac{|C_\mathbf{i}|}{N-\mathbf{i}}=\frac{\|\partial_d \alpha\|}{\|\alpha\|} \geq h_d.
\end{align*}
We can now prove our inequality. All the sums on $d$-cells are on $X^d \cup X_\partial^d$ and are taken with multiplicity.
\begin{align}
\lambda_d 	&=	\frac{\|\partial_d f \|_2^2}{\|f\|_2^2}	\nonumber \\
			&= 	\frac{\sum_\mu \partial_d f(e_\mu^{d-1})^2}{\sum_\lambda f(e_\lambda^d)^2}	\nonumber\\
			&=	\frac{\sum_{e_i^d \sim e_j^d} \left(f(e_i^d\right) \pm f(e_j^d))^2}{\sum_\lambda f(e_\lambda^d)^2} \cdot \frac{\sum_{e_i^d \sim e_j^d}\left(f(e_i^d) \mp f(e_j^d)\right)^2}{\sum_{e_i^d \sim e_j^d}\left(f(e_i^d) \mp f(e_j^d)\right)^2}	\label{boundary}\\ 
			&	\geq \frac{\left(\sum_{e_i^d \sim e_j^d}|f(e_i^d)^2-f(e_j^d)^2| \right)^2}{\left(\sum_\lambda f(e_\lambda^d)^2 \right) \cdot \left(\sum_{e_i^d \sim e_j^d}(f(e_i^d) \mp f(e_j^d))^2 \right)}	\label{cauchy_schwartz}\\
			&\geq \frac{\left( \sum_{e_i^d \sim e_j^d}|f(e_i^d)^2-f(e_j^d)^2| \right)^2}{\left(\sum_\lambda f(e_\lambda^d)^2 \right) \cdot 2 \left(\sum_{e_i^d \sim e_j^d}f(e_i^d)^2 + f(e_j^d)^2 \right)} \nonumber\\
			&\geq \frac{\left( \sum_{e_i^d \sim e_j^d}|f(e_i^d)^2-f(e_j^d)^2| \right)^2}{\left(\sum_\lambda f(e_\lambda^d)^2 \right) \cdot \left( 2 m \sum_\lambda f(e_\lambda^d)^2 \right)}	\nonumber\\
			&= \frac{\left( \sum_{i=0}^{N-1}(f(e_{i+1}^d)^2-f(e_i^d)^2) |C_i|\right)^2}{2m \left(\sum_\lambda f(e_\lambda^d)^2 \right)^2} \label{Ci}	\\
			&\geq \frac{\left( \sum_{i=0}^{N-1}(f(e_{i+1}^d)^2-f(e_i^d)^2) H[f](N-i)\right)^2}{2m \left(\sum_\lambda f(e_\lambda^d)^2 \right)^2} \nonumber	\\
			&= \frac{H[f]^2}{2m} \cdot \frac{\left( \sum_\lambda f(e_\lambda^d)^2 \right)^2}{\left( \sum_\lambda f(e_\lambda^d)^2 \right)^2}	\nonumber\\
			& \geq \frac{h_d^2(X)}{2m} \nonumber.
\end{align}
The equality \eqref{boundary} is a consequence of $f|_{X^d_\partial}=0$ and $\partial_d f =0$ for a $(d-1)$-cell of degree $0$ and   \eqref{cauchy_schwartz} follows from Cauchy-Schwartz. For \eqref{Ci}, we want to show that 
\begin{equation*}
\sum_{e_i^d \sim e_j^d}|f(e_i^d)^2-f(e_j^d)^2| = \sum_{i=0}^{N-1}(f(e_{i+1}^d)^2-f(e_i^d)^2) |C_i|.
\end{equation*}
This can be seen by counting the number of times each $f(e_i^d)^2$ appears on each side. On the left, $f(e_i^d)^2$ appears 
\begin{align*}
\Sigma_i \coloneqq & \left| \left\{ \{e_i^d, e_j^d \} : j<i \text{ and } e_i^d \sim e_j^d \text{ with multiplicity}\} \right\} \right| \\*
-& \left| \left\{\{e_i^d, e_k^d \} : i<k \text{ and } e_i^d \sim e_k^d \text{ with multiplicity}\} \right\}\right|.
\end{align*}
On the other side, each $f(e_i^d)^2$ appears $|C_{i-1}| - |C_{i}|$ times. Remark that for $j<k$ such that $\{e^d_j,e_k^d \}$ is in $C_{i-1}$, $\{e^d_j,e_k^d \}$ is also in $C_i$ if $k \neq i$.  Similarly, $\{e^d_j,e_k^d \}$ in $C_{i}$ is also in $C_{i-1} $ if $j\neq i$. Then,
\begin{equation*}
|C_{i-1}|-|C_i| = \Sigma_i
\end{equation*}
\end{proof}

\begin{center}{\textbf{Acknowledgments}}
\end{center}
The author is thankful to P. de la Harpe and T. Nagnibeda for their useful comments and careful reading of previous versions of this text.  He is also grateful to the anonymous referees for their valuable remarks.
\vskip 0.4 true cm


\bibliography{expenseurs.bib}
\bibliographystyle{abbrv}

\bigskip
\bigskip

{\footnotesize \pn{\bf Grégoire Schneeberger}\; \\ {Section de mat\'ematiques}, {University of Geneva}\\
{\tt Email: gregoire.schneeberger@unige.ch}\\
\end{document}